\documentclass[11pt]{amsart}
\usepackage{geometry}                
\usepackage{graphicx}
\usepackage{amssymb}
\usepackage{epstopdf}

\usepackage{csquotes}  
\usepackage{mathrsfs}

\usepackage{hyperref} 
\hypersetup{
 colorlinks   = true,
 urlcolor     = blue,
 linkcolor    = blue,
 citecolor   = red ,
 bookmarksopen=true
}

\newtheorem{theorem}{Theorem}[section]
\theoremstyle{plain}

\newtheorem{definition}[theorem]{Definition}

\newtheorem{proposition}[theorem]{Proposition}
\newtheorem{remark}[theorem]{Remark}

\numberwithin{equation}{section}
\newcommand{\R}{\mathbb{R}}
\newcommand{\pa}{\partial}
\newcommand{\Om}{\Omega}

\newcommand{\cL}{\mathcal{L}}

\newcommand{\La}{\mathscr{L}}
\newcommand{\ve}{\varepsilon}
\newcommand{\cH}{\mathscr{H}}
\newcommand{\vf}{\varphi}

\newcommand{\B}{\mathscr B}

\newcommand{\Sc}{\mathscr S}

\newcommand{\Ls}{\left(-\La\right)}

\newcommand{\goto}{\rightarrow}
\newcommand{\Rn}{\R^n}
\newcommand{\Sn}{\mathscr{S}(\Rn)}
\newcommand{\la}{\lambda}
\newcommand{\p}{\partial}

\title[The Harnack inequality for a class of nonlocal, etc.]{The Harnack inequality for a class of nonlocal parabolic equations}

\author{Agnid Banerjee}
\address{TIFR CAM, Bangalore-560065} \email[Agnid Banerjee]{agnidban@gmail.com}
 \thanks{First author is supported in part by SERB Matrix grant MTR/2018/000267}

\author{Nicola Garofalo}
\address{Dipartimento di Ingegneria Civile, Edile e Ambientale (DICEA) \\ Universit\`a di Padova\\ 35131 Padova, ITALY}
\email[Nicola Garofalo]{nicola.garofalo@unipd.it}

\thanks{Second author was supported in part by a Progetto SID (Investimento Strategico di Dipartimento) ``Non-local operators in geometry and in free boundary problems, and their connection with the applied sciences", University of Padova, 2017.}

\author{Isidro H. Munive}
\address{Department of Mathematics, University Center of Exact Sciences and Engineering, University of Guadalajara, 44430 Guadalajara, Mexico}\email[Isidro Munive]{imunivel@gmail.com }
\thanks{Third author is supported by CONACYT grant 265667, Instituto de Matem\'aticas, UNAM}

\author{Duy-Minh Nhieu}
\address{Department of Mathematics, National Central University, Zhongli District, Taoyuan city 32001, Taiwan, R.O.C.}\email[Duy-Minh Nhieu]{dmnhieu@math.ncu.edu.tw}
\thanks{Fourth author supported by the ministry of Science and Technology, R.O.C., Grant MOST 108-2115-M-008-011}

\begin{document}
\maketitle

\tableofcontents

\begin{abstract}
In this paper we establish a scale invariant Harnack inequality for the fractional powers of parabolic operators $(\p_t - \La)^s$, $0<s<1$, where $\La$ is the infinitesimal generator of a class of symmetric semigroups. As a by-product we also obtain a similar result for the nonlocal operators $\Ls^s$. Our focus is on non-Euclidean situations. 
\end{abstract}

\section{Introduction}
In Euclidean spaces the fractional powers of the Laplacian, $(-\Delta)^s$,  have been studied in a variety of contexts, ranging from isotropic diffusion with jumps to thin obstacle problems and phase transition problems. See for instance \cite{CSS}, \cite{SV}, \cite{CSire}, and also \cite{Gft} for an up-to-date overview of recent developments. The Harnack inequality for the nonlocal Laplacian was first established by Landkof in \cite{La} using potential methods. One should also see the contributions of K. Bogdan, see \cite{Bo} and subsequent papers by the same author with various coauthors, and also those of Kassmann, see \cite{Ka1, Ka2}. In their 2007 celebrated work \cite{CS} Caffarelli and Silvestre gave an alternative proof based on the so-called extension method. They showed that a solution to the nonlocal equation $(-\Delta)^s u =0$ in an open set $\Om \subset \R^n$ can be lifted to a function $U$ in $\Om\times (0,\infty)\subset \R^{n+1}$ which solves a degenerate elliptic equation  satisfying the structural assumptions in the work of Fabes, Kenig and Serapioni \cite{FKS}. This allowed them to obtain  the Harnack inequality for $(-\Delta)^s$ from the De Giorgi-Nash-Moser theory developed in \cite{FKS} of such degenerate elliptic equations.

The focus of the present paper are situations that go beyond the Euclidean setting recalled above. More precisely, we prove a strong Harnack inequality for a class of degenerate parabolic equations  which arise in the extension problem for nonlocal operators $\cH^s=(\partial_t - \La)^s$, $0<s<1$, where $\La$ is, for instance, a locally subelliptic operator. A similar inequality for the time-independent case easily descends from this result. Similarly to what was done in \cite{CS}, with such Harnack inequalities we then obtain corresponding ones for $\cH^s$ and $(-\La)^s$. Some basic scenarios to which, under the hypothesis specified in Section \ref{pl}, our results apply are as follows: 
\begin{itemize}
\item[(i)] $\La$ is a sub-Laplacian on a Carnot group $\mathbb{G}$, see \cite{Fo};
\item[(ii)] $\La$ is a sub-Laplacian generated by a system of vector fields satisfying H\"ormander's finite rank condition, see \cite{Hor};
\item[(iii)] $\La$ is a locally sub-elliptic operator in the sense of Fefferman and Phong, see \cite{FP};
\item[(iv)] $\La$ is the infinitesimal generator of a strongly continuous symmetric semigroup on $L^2$. 
\end{itemize}
Given the operator $\La$ we consider the associated heat operator $\cH = \p_t - \La$.
Using the  spectral resolution of $-\La$, and the Fourier transform in the time variable $t$, we first introduce the relevant function spaces $W^{2s}$ and their parabolic counterpart $ H^{2s}$. These spaces constitute the natural domains for the fractional powers $(-\La)^s$ and $\cH^s$ respectively, and they generalise the time-independent fractional Folland-Stein Sobolev space $W^{2s,2}(\mathbb{G})$ defined in \cite{Fo} in the setting of  Carnot groups. We then show that for a ``Dirichlet" datum $u \in H^{2s}$ the extension problem for $\cH^s$ is well-defined. Moreover, the corresponding solution belongs to the right ``energy space". The following is the central result in this direction.

\begin{theorem}\label{wk2}
Given a function $u \in H^{2s}$, let $V$ be defined by
\begin{equation}\label{v}
V(x,z,t) =
\int_0^\infty \int_{\R^n}  P^{(a)}_z(x,y,\tau) u(y,t - \tau) dy d\tau, 
\end{equation} 
where the function $P^{(a)}_z(x,y,t)$ denotes the \emph{Poisson kernel} for the extension operator defined in \eqref{extKernel} below.
Then, $V$ solves the degenerate parabolic partial differential equation in \eqref{ep}. Furthermore, we have 
\begin{equation}\label{tr2}
\lim_{z\rightarrow 0^+}\|V(\cdot,z,\cdot)-u\|_{L^{2}\left(\R^{n+1}\right)}=0,
\end{equation}
and
\begin{equation}\label{dn2}
\lim_{z\rightarrow 0^+} \left\|\frac{2^{-a}\Gamma\left(\frac{1-a}{2}\right)}{\Gamma\left(\frac{1+a}{2}\right)}z^a\frac{\pa V}{\pa z}(\cdot,z,\cdot) + \cH^{s}u\right\|_{L^{2}\left(\R^{n+1}\right)} = 0.
\end{equation}
Finally, given $M>0$, and with $\Sigma= \R^n \times (0, M) \times \R$, there exists $C=C(a,\Sigma)>0$ such that 
\begin{equation}\label{en2}
\|V\|_{\cL^{1,2}(\Sigma, z^a dx dz dt)} \leq  C\|u\|_{\mathcal{H}^{2s}}.
\end{equation} 
\end{theorem}

For the definition of the Sobolev space $\cL^{1,2}(\Sigma, z^a dx dz dt)$, one should see \eqref{ss} below. Once Theorem \ref{wk2} is established, we are able to obtain the Harnack inequality for the nonlocal operator $\cH^s$ in the thin space from the regularity results of the corresponding extension operator $\cH_a$ in the thick space. In this perspective, the analysis becomes at this point similar to the above described approach for the fractional Laplacian in \cite{CS} based on the De-Giorgi-Nash-Moser theory developed in the seminal work  \cite{FKS} of Fabes, Kenig and Serapioni, and its parabolic counterpart \cite{CSe} by Chiarenza and Serapioni. Our main result is the following.

\begin{theorem}\label{parharnack} 
Let $0<s<1$, and $u \in H^{2s}$, with $u \geq 0$ in $\R^n \times \R$, be a solution to $\cH^{s}u = 0$
in $B(x, 2r) \times (-2r^2, 0]$. Then, the following Harnack inequality holds
\begin{equation}\label{parhar1}
\underset{B(x,r) \times (-r^2, -\frac{r^2}{2}]}{\sup}\ u  \leq  C \underset{B(x, r) \times (-\frac{r^2}{4},0]}{\inf}\  u,
\end{equation}
for some universal $C>0$. Here $B(x,r)$ denotes  the relative metric ball  of radius $r$ centered at $x$ associated with the \emph{carr\'e du champ} corresponding to $\mathscr{L}$ (see Section \ref{pl} for the precise notion).
\end{theorem}

Theorem \ref{parharnack} implies the following corresponding Harnack inequality for nonnegative solutions of the nonlocal operator $(-\La)^s$.

\begin{theorem}\label{harnack}
Let $s \in (0, 1)$ and  let $u \in W^{2s}$, with $u \geq 0$ in $\R^n$, be a solution in $B(x,2r)$ to $\Ls^{s} u  = 0$.
Then, there exists a universal $C>0$ such that 
\begin{equation}\label{main}
\sup_{B(x,  r)} u  \leq  C \inf_{B(x,r)}  u.
\end{equation}
\end{theorem}

We mention that Theorem \ref{harnack} represents a generalisation of the result first established for sub-Laplacians in Carnot groups by Ferrari and Franchi in \cite{FF}. We also mention that for the standard fractional heat operator $(\p_t - \Delta)^s$ in Euclidean space a strong Harnack inequality was recently proved by the first and second named authors in \cite{BaG}. Their approach was based on the extension problem for the fractional heat equation that, for $s = 1/2$, was first introduced by F. Jones in his pioneering paper \cite{Jo}, and more recently developed for all $s\in (0,1)$ independently by Nystr\"om and Sande \cite{NS}, and Stinga and Torrea \cite{ST}.

The present work is organised as follows.  In Section \ref{pl} we introduce some basic notations and gather some known results which are relevant to our work in the subsequent sections. We then recall the notions of $\cH^s$ and $\Ls^s$ developed in \cite{G} based on Balakrishnan's calculus. Combining the spectral resolution for $- \La$ with the Fourier transform in the $t$ variable we also define the relevant function spaces  $W^{2s}$, and their parabolic counterpart $H^{2s}$. In Section \ref{et1} we recall the extension problem for $\cH^s$ (following \cite{G}, the one for $\Ls^s$ is obtained from the latter using Bochner's subordination). We then prove our main result, Theorem \ref{wk2}, which shows that such problem is well-defined in the parabolic fractional  Sobolev spaces $H^{2s}$. Finally, in Section \ref{final} we establish the Harnack inequality for $\cH^s$, using the ideas developed in \cite{CSe}, \cite{Gry}, \cite{Sa} and \cite{CS}. 


\section{Preliminaries}\label{pl}

In this section we collect some known results that will be used throughout our work. We emphasise that although we will confine our attention to the setting of smooth sub-Laplacians, we could have worked in the more abstract setting of symmetric Dirichlet forms supporting a doubling condition and a Poincar\'e inequality without essential changes. 
In what follows we consider $\Rn$ endowed with Lebesgue measure (alternatively, we could have worked with a smooth measure on $\Rn$ and our results would hold unchanged) and with a system of globally defined smooth vector fields $X_1,...,X_m$. We indicate with $X_i^\star$ the formal adjoint of $X_i$ in $L^2(\Rn)$, and with $\La = - \sum_{i=1}^m X_i^\star X_i$ the sub-Laplacian associated with the $X_i$'s. We assume that $\La$ be locally subelliptic in the sense of \cite{FP}. The relative \emph{carr\'e du champ} will be denoted by 
\begin{equation}\label{cdushamp}
|Xu|^2 = \frac 12 (\La(u^2) - 2 u \La u).
\end{equation}
 We note that, indicating with $(\cdot,\cdot)$ the inner product in $L^2(\Rn)$, we have for $u, v\in \Sn$
\[
(-\La u,v) = (u,-\La v),\ \ \ \ \ \ \ \ \ (-\La u,u) = \int_{\Rn} |Xu|^2 \ge 0,
\]
thus the operator $-\La$ is symmetric and positive. We define
\begin{equation}\label{d}
d(x,y) = \sup\{|\vf(x) - \vf(y)|\mid \vf\in C^\infty(\Rn),\ |X\vf|\le 1\},
\end{equation}
see \cite{BM, BM2} and \cite{Davies}.
Throughout the paper we assume that $d(x,y)$ defines a distance on $\Rn$. We also assume that the metric topology is equivalent to the pre-existing Euclidean one in $\Rn$, see \cite[Assumption A]{Sturm}, and also \cite{GN} for a discussion of this aspect. The relative balls will be denoted by $B(x,r)$, and we let $V(x,r) = |B(x,r)|$, where $|E|$ indicates the Lebesgue measure of the set $E$. 
We will assume that there exists a constant $C_D>0$ such that for all $x\in \R^n$ and $r>0$ one has
\begin{equation}\label{doubling}
V(x,2r) \leq  C_D V(x,r).
\end{equation}
A third basic hypothesis that we make is that there exists another constant $C_P>0$ such that for every ball $B(x,r)$, and every function $u\in C^{0,1}(B(x,2r))$, one has
\begin{equation}\label{poi}
\int_{B(x,r)} |u - u_B|^2 dx \le C_P\ r^2 \int_{B(x,2r)} |Xu|^2 dx. 
\end{equation}
Under these assumptions it is well-known from \cite{Gry} and \cite{Sa} that one has the Harnack inequality for the heat operator $\cH = \frac{\pa}{\pa t} - \La$. In particular, the heat semigroup 
\begin{equation}\label{pt}
P_t f(x) = e^{-t \La} f(x) = \int_{\Rn} p(x,y,t) f(y) dy,
\end{equation}
 has a strictly positive kernel $p(x,y,t)$ satisfying the Gaussian bounds:
there exist constants $C, \alpha, \beta >0$ such that for every $x,y\in \Rn$
and $t>0$ one has
\[
\frac{C}{V(x,\sqrt
t))^{\frac{1}{2}}V(y,\sqrt t))^{\frac{1}{2}}} e^{- \frac{\alpha d(x,y)^2}{t}} \le p(x,y,t)\le \frac{C^{-1}}{V(x,\sqrt
t))^{\frac{1}{2}}V(y,\sqrt t))^{\frac{1}{2}}} e^{- \frac{\beta d(x,y)^2}{t}}.
\]
Before proceeding we note that, under our assumptions, the semigroup is stochastically complete, i.e., $P_t1 = 1$. This means  that for every $x \in \R^n$ and $t > 0$ one has 
\begin{equation}
\label{stc}
\int_{\R^n}p(x,y,t) dy = 1.
\end{equation}
This follows from the fact that \eqref{doubling} implies, in particular, that for every $x\in \Rn$ and every $r\ge 1$ one has
$V(x,r) \le r^Q V(x,1)$,
where $Q = \log_2 C_D$. This estimate implies that $\int_1^\infty \frac{r}{\log V(x,r)} dr = \infty$. We can thus appeal to \cite[Theor. 4]{Sturm} to conclude that $P_t 1 = 1$ (see also \cite{Gri} for the Riemannian case).
A first basic example under which the above assumptions are satisfied is of course that of a Carnot group. In such setting, because of the presence of non-isotropic dilations the hypothesis \eqref{doubling} is trivially satisfied, whereas \eqref{poi} was established in \cite{Var}. More in general, one can consider a system $X_1,...,X_m$ of smooth vector fields satisfying H\"ormander's finite rank condition in \cite{Hor} in a bounded open set $\Om\subset \Rn$. In such situation, it is well known from the fundamental works \cite{NSW} and \cite{J} that \eqref{doubling} and \eqref{poi} hold locally. One can turn these local results into global ones by using the tools developed in \cite{BBLU}. Specifically, consider a compact set $K\supset \bar{\Omega}$ in which the vector fields $X_1,...,X_m$ are defined. It is possible to extend such system to the whole $\Rn$ in such a way that outside  $K$  the vector fields become the standard basis of $\Rn$, so that $\La=\Delta$, see \cite[Theorem 2.9]{BBLU}. Moreover, the new control distance associated with the extended vector fields satisfies \eqref{doubling} globally. As a consequence of what we have said above, the semigroup $P_t$ satisfies \eqref{stc}, see also \cite[(3.2) in Theor. 3.4]{BBLU} for a direct proof.

The semigroup $P_t$ defines a family of bounded operators $P_t : L^2(\R^n) \goto L^2(\R^n)$ having the following properties:
\begin{enumerate}
\item  $P_0 =\mathrm{Id}$ and for $s,t\geq 0,P_{s+t} = P_s\circ P_t$;
\item for $u \in L^2(\R^n)$, we have
$\|P_tu\|_{L^2(\R^n)} \ \leq\  \|u\|_{L^2(\R^n)}$;
\item for $u \in L^2(\R^n)$, the map $t \goto P_tu$ is continuous in $L^2(\R^n)$;
\item for $u, v \in L^2(\R^n)$ one has
\[
\int_{\R^n}(P_tu)\,v\,dx\ =\ \int_{\R^n}u(P_tv)\,dx.
\]
\end{enumerate}
Properties (1)-(4) above can be summarized by saying that $\{P_t\}_{t\geq 0}$ is a self-adjoint strongly continuous contraction semigroup on $L^2(\R^n)$. From the spectral decomposition, it is also easily checked that the operator $\mathscr{L}$ is furthermore the generator of this semigroup, that is for $u \in \operatorname{Dom}(\La)$ (the domain of $\mathscr{L}$), one has
\begin{equation}
\label{decayl}
\lim_{t\goto 0^+}\Big\|\frac{P_tu-u}{t}-\mathscr{L}u\Big\|_{L^2(\R^n)} = 0.
\end{equation}
This implies that for $t \geq 0, P_t \operatorname{Dom}(\La) \subset \operatorname{Dom}(\La)$, and that for $u \in \operatorname{Dom}(\La)$,
\[
\frac{d}{dt} P_tu  =  P_t\mathscr{L}u  = \mathscr{L}P_tu, 
\]
the derivative in the left-hand side of the above equality being taken in $L^2(\R^n)$. Formula \eqref{decayl} shows in particular that for any $u \in \operatorname{Dom}(\La)$, one has
\begin{equation}
\label{decay}
\|P_tu - u\|_{L^2(\R^n)}  =  O(t)\quad \text{as}\ t \goto 0^+.
\end{equation}

We next recall the fractional calculus of the operators $\La$ and $\mathscr H = \frac{\pa}{\pa t} - \La$ introduced in \cite{G}. For related developments covering a quite different class of H\"ormander type equations, we refer the reader to the recent work \cite{GT} by Tralli and the second named author.

\begin{definition}\label{fracl}
Let $s\in (0,1)$. For any $u \in \Sc(\R^n)$ we define the nonlocal operator 
\begin{eqnarray}
\label{fracls}
\Ls^{s}u(x) = -\frac{s}{\Gamma(1-s)}\int^{\infty}_0 t^{-s-1}[P_tu(x) -u(x)] dt.
\end{eqnarray}
\end{definition}

In an abstract setting formula \eqref{fracls} is due to Balakrishnan, see \cite{B1} and \cite{B2}. 
The integral in the right-hand side of \eqref{fracls} must be interpreted as a Bochner integral in $L^2(\R^n)$. We note explicitly that, in view of \eqref{decay} and of (2) above, such integral is convergent in $L^2(\R^n)$ for every $u \in \operatorname{Dom}(\La)$, and thus in particular for every $u \in \Sc(\R^n)$. We emphasise that, as observed in \cite[Lemma 8.5]{G}, Definition \ref{fracl} coincides with that proposed in \cite{FF} in the setting of Carnot groups.  
To define now the fractional power $\cH^s$, we recall the notion of evolutive semigroup  
\[
P^{\cH}_{\tau}u(x,t) = \int_{\R^n}p(x,y,\tau) u(y,t-\tau) dy,
\]
for whose properties we refer the reader to \cite{G}. Observe that when $u(x,t) = v(x)$, then $P^{\cH}_{\tau}u(x,t) = P_t v(x)$. Furthermore, $P^{\cH}_{\tau}$ is contractive in $L^2(\R^{n+1})$,
\begin{equation}\label{conL2}
\|P^{\cH}_{\tau} u\|_{L^{2}(\R^{n+1})} \leq  \|u\|_{L^{2}(\R^{n+1})},
\end{equation}
and, analogously to \eqref{decay}, for any $u\in \operatorname{Dom}(\mathscr H)$ (and thus in particular for $u \in \Sc(\R^{n+1})$), we have 
\begin{equation}\label{rateH}
\|P^{\cH}_{\tau} u-u\|_{L^2(\R^{n+1})} =  O(\tau).
\end{equation}

\begin{definition}
Given $s\in (0,1)$, and  a function $u \in \Sc(\R^{n+1})$, similarly to \eqref{fracls} we define
\begin{equation}
\label{fracsh}
\cH^s u(x,t)\overset{def}{=}  -\frac{s}{\Gamma(1-s)}\int^{\infty}_0\tau^{-s-1}\Big[P^{\cH}_{\tau}u(x,t)-u(x,t)\Big] d\tau.
\end{equation}
\end{definition}
Using \eqref{conL2} and \eqref{rateH}, it is easy to see that definition \eqref{fracsh} is well-posed in $L^2(\R^{n+1})$.  
Next, we denote by $\{E(\la)\}$ the spectral measures associated with $\La$. Let
\begin{equation}\label{spectralD}
\La = - \int^{\infty}_0\lambda dE(\lambda)
\end{equation}
denote the spectral resolution of $\La$ (see e.g. \cite[Chap. 8]{Y}, or \cite[Chap. 27]{Se} for a detailed discussion). In terms of the spectral measures the heat semigroup $\{P_t\}$ is given by
\begin{equation}\label{hs}
P_t = \int^{\infty}_0e^{-\lambda t} dE(\lambda).
\end{equation}
An expression for  the fractional powers of the sub-Laplacian in terms of this decomposition is then given by
\begin{equation}\label{fracsl}
\Ls^{s}u = \int^{\infty}_0\lambda^{s} dE(\lambda)u, \quad \text{for $u\in \Sc(\R^n)$}.
\end{equation}
The identity \eqref{fracsl} is easily verified using the well-known integral 
\[
\int_0^\infty \tau^{-s -1} (1-e^{-\tau}) d\tau = \frac{\Gamma(1-s)}{s},\ \ \ \ \ \ \ \  \ 0<s<1,
\]
which gives $- \frac{s}{\Gamma(1-s)} \int_{0}^{\infty} \frac{e^{-\lambda t} -1 }{t^{1+s}} dt = \lambda^{s}$. This formula continues to be valid for any complex number $w = \la+2\pi i \sigma\in \mathbb C$, with $\la>0$. We record this for later use
\begin{equation}\label{la}
- \frac{s}{\Gamma(1-s)} \int_{0}^{\infty} \frac{e^{-(\la+2\pi i \sigma) t} -1 }{t^{1+s}} dt = (\la+2\pi i \sigma)^{s},\ \ \ \ \ \ \ \la>0, \sigma\in \R.
\end{equation}
Using \eqref{la} with $\sigma = 0$, the spectral representation \eqref{hs}
and Balakrishnan's formula \eqref{fracls}, we find
\begin{align*}
\Ls^{s}u(x) & = -\frac{s}{\Gamma(1-s)}\int^{\infty}_0 t^{-s-1}[P_tu(x) -u(x)] dt
\\
& = -\frac{s}{\Gamma(1-s)}\int^{\infty}_0 t^{-s-1}[\int_{0}^{\infty} e^{-\lambda t} dE(\lambda) u(x) -u(x)] dt\notag
\\
& = \int_{0}^{\infty} \lambda^s d E(\lambda) u(x),
\end{align*}
where in the last equality we have exchanged the order of integration.  Now, for a function $u\in\Sc(\R^n)$ and $s \in (0,1)$ we define the norm
\begin{equation}\label{norm1}
\|u\|_{2s} \overset{def}{=} \left(\int^{\infty}_0\left(1+\lambda^{s}\right)^2 d\|E(\lambda)u\|^2\right)^{\frac{1}{2}},
\end{equation}
where we have used the fact that $E(\la)$ is idempotent and symmetric to write $(E(\lambda)u,u) = (E(\lambda)u,E(\lambda)u) = \|E(\lambda)u\|^2$. We then let
\begin{equation}
\label{spacew}
W^{2s} = W^{2s}(\R^n) \overset{def}{=} \overline{\Sc(\R^{n})}^{\|\cdot\|_{2s}}.
\end{equation}
From \eqref{fracsl} we thus see that the definition of $\Ls^s$ is extendible to functions in  $W^{2s}$. We stress  that in the classical setting when $\La=\Delta$, the space $W^{2s}$ coincides with the fractional Sobolev space  of order $2s$  defined by the norm
\begin{equation}\label{ed}
\bigg(\int_{\R^n} \left( u^2 + ((-\Delta)^s u)^2  \right) \bigg)^{1/2}.
\end{equation}
This is easily seen from \eqref{fracsl} since in terms of the spectral resolution of $-\Delta$, we have that 
\[
\int  u^2 +  ((-\Delta)^s u)^2 = \int_0^{\infty} (1+ \lambda^{2s}) d \|E(\lambda)u\|^2 \cong \int^{\infty}_0\left(1+\lambda^{s}\right)^2 d\|E(\lambda)u\|^2.
\]
Thus, the norm in \eqref{ed} is finite if and only if the corresponding norm in \eqref{norm1} is finite.
 We also mention that in the classical case when $0< s< 1/2$,  the space  $W^{2s}$ coincides with the fractional Aronszajn-Gagliardo-Slobedetzky space of the functions $f\in L^2(\Rn)$ such that \[\int_{\Rn} \int_{\Rn} \frac{|f(x) - f(y)|^2}{|x-y|^{n+4s}} dx dy<\infty.\]

Now for a function $f\in L^1(\R)$ we denote its Fourier transform by
\[
\hat{f}(\sigma) = \int_{\R}e^{-2\pi i\sigma t}f(t) dt, \quad \sigma\in\R.
\]
Let $u\in \Sc(\R^{n+1})$, then  by taking the Fourier transform with respect to the time variable, we easily find
\begin{equation}\label{ptauH}
\widehat{P^{\cH}_{\tau} u}(x,\sigma) = e^{-2\pi i \tau\sigma} P_\tau(\hat u(\cdot,\sigma))(x).
\end{equation}
From this observation, in view of \eqref{la},  we obtain
\begin{eqnarray}\label{he}
\widehat{\cH^{s}u}(x,\sigma)& = & -\frac{s}{\Gamma(1-s)}\int^{\infty}_0\frac{1}{\tau^{1+s}}\Big[\int^{\infty}_0\left(e^{-\left(\lambda+2\pi i \sigma\right) \tau}-1\right) dE(\lambda)\hat{u}(\cdot,\sigma)\Big] d\tau\\
& = & \int^{\infty}_0\left(\lambda+2\pi i \sigma\right)^{s} dE(\lambda)\hat{u}(\cdot,\sigma).
\notag
\end{eqnarray}
Hence,
\begin{equation}\label{hes}
\|\widehat{\cH^{s}u}(\cdot,\sigma)\|^2_{L^2(\R^n)} = \int^{\infty}_0|\lambda+2\pi i \sigma|^{2s} d\|E(\lambda)\hat{u}(\cdot,\sigma)\|^2.
\end{equation}
Similarly to the stationary case, for  a function $u\in\mathscr{S}(\R^{n+1})$ and $s \in (0,1)$ we  now define the following norm
\begin{equation}\label{H2s}
\|u\|_{\mathcal{H}^{2s}}  \overset{def}{=} \left(\int^{\infty}_{-\infty}\int^{\infty}_0\left(1+|\lambda+2\pi i \sigma|^{s}\right)^2 d\|E(\lambda)\hat{u}(\cdot,\sigma)\|^2 d\sigma\right)^{\frac{1}{2}}.
\end{equation}
We then let
\begin{equation}
\label{spaceh}
H^{2s} = \overline{\Sc(\R^{n+1})}^{\|\cdot\|_{\mathcal{H}^{2s}}}. 
\end{equation}
In view of the computation in \eqref{he}, we see that  the definition of $\cH^s$ is extendible to functions in  $H^{2s}$. As above, it can be seen that in the classical setting when $\cH=\p_t - \Delta$, the space $H^{2s}$ coincides with the space $\operatorname{Dom}(\cH^s)$ adopted in \cite{ST} and \cite{BaG}. 


\section{Extension problem for $(-\La)^s$ and $\cH^s$}\label{et1}
In this section, given $s\in (0,1)$, we indicate with $a = 1-2s$. Note that $-1<a<1$, and that $a = 0$ when $s = 1/2$. To discuss  the extension problem for $(-\La)^s$ and $\cH^s$, following \cite{G} we begin with that for $\cH^s$, and then use Bochner subordination for the time-independent case. We will use the letter $z$ for the so-called extension variable, and denote by $\B_a = \frac{\partial^2}{\partial z^2}+\frac{a}{z}\frac{\partial}{\partial z}$ the Bessel operator on the half-line $\{z>0\}$. Given the heat operator $\cH = \p_t - \La$ with respect to the variables $(x,t)\in \R^{n+1}$, the extension problem for $\cH^s$ consists in solving the ``Dirichlet problem" in $\R^{n+2}_+$
\begin{equation}\label{ep}
\begin{cases}
\cH_aV \overset{def}{=} z^{a}\left(\cH V+\B_aV\right) = 0,
\\
V(x,0,t) = u(x,t).
\end{cases}
\end{equation}
We note that the  operator $\cH_a$ can be written as a degenerate parabolic operator of the form
\begin{equation}
\label{extdiv}
\cH_a = \omega(x,z)\frac{\pa}{\pa t}+\sum^{m+1}_{i=1}\tilde{X}^{\star}_i\left(\omega(x,z)\tilde{X}_i\right),
\end{equation}
where 
\begin{equation}\label{tildes}
\tilde{X}_1 = X_1,\ ...\ , \tilde{X}_m = X_m,\ \tilde{X}_{m+1} = \frac{\pa}{\pa z}, \quad \text{and}\quad \omega(x,z) = |z|^a.
\end{equation}
Given an open set $\Sigma \subset \Rn\times \R_+ \times \R$, the Sobolev space $\cL^{1,2}(\Sigma, z^a dx dz dt)$ is defined as the collection of all functions $U\in L^{2}(\Sigma, z^a dx dz dt)$ such that $XU, \p_z U\in L^{2}(\Sigma, z^a dx dz dt)$. Such space is endowed with its natural norm
\begin{equation}\label{ss}
||U||_{\cL^{1,2}(\Sigma, z^a dx dz dt)} = \bigg(\int_{\Sigma} \bigg(|U|^2 + |XU|^2 + |\p_z U|^2\bigg)z^a dx dz dt\bigg)^{1/2}.
\end{equation}
We now introduce the relevant notion of weak solution.

\begin{definition}\label{wk0}
Given $ \Sigma= B(x, r) \times (0, M) \times (T_1, T_2)$, a function $W \in \cL^{1,2}( \Sigma, z^a dx dz dt)$ is said to be a weak solution in $\Sigma$ of
\begin{equation}
\begin{cases}
\cH_a W=0.
\\
\underset{z \to 0^+}{\lim} z^a \partial_z W= \psi,\quad \text{ for $\psi \in L^{2} (B(x,r) \times (T_1, T_2), dxdt)$}
\end{cases}
\end{equation}
if for every $\phi \in C^{\infty}( B(x,r) \times [0, M] \times [T_1, T_2])$ with compact support in $B(x,r) \times [0, M) \times [T_1, T_2]$, and almost every $t_1 , t_2$ such that $T_1 <t_1 <t_2<T_2$, we have
 \begin{align}\label{d8}
& \int_{t_1}^{t_2} \left(\int_{B(x,r) \times (0, M)} \sum_{i=1}^{m+1}  \tilde X_i W \tilde X_i\phi |z|^a dxdz\right) dt= \int_{t_1}^{t_2} \left(\int_{B(x,r) \times (0, M)} |z|^a \phi_t W dxdz\right) dt
\\
& - \int_{B(x,r) \times (0, M)} \phi (\cdot,t_2) W (\cdot, t_2) |z|^a dxdz   + \int_{B(x,r) \times (0, M)} \phi (\cdot,t_1) W (\cdot,t_1) |z|^a  dxdz
\notag
\\
& - \int_{t_1}^{t_2} \int_{B(x,r) \times \{0\}}   \psi \phi dx dt.
\notag
 \end{align}
 \end{definition}

Hereafter, $p(x,y,t)$ indicates the heat kernel in \eqref{pt}. The following \emph{Poisson kernel} for the problem \eqref{ep} was introduced in \cite{G}:
\begin{equation}\label{extKernel}
P^{(a)}_z(x,y,t)\overset{def}{=} \frac{1}{2^{1-a}\Gamma\left(\frac{1-a}{2}\right)}\frac{z^{1-a}}{t^{\frac{3-a}{2}}}\,e^{-\frac{z^2}{4t}}\,p(x,y,t).
\end{equation}
For the main properties of the function $P^{(a)}_z(x,y,t)$, see \cite[Prop. 9.5, 9.6, 10.3, 10.4]{G}. Here, we reproduce the one that represents a key ingredient in solving \eqref{ep}.

\begin{proposition}\label{G-prop 9.6}
For every $x, y \in \R^n$ with $x \neq y$ and $t > 0$ one has
\[
\partial_t P^{(a)}_z(x,y,t) - \B_a P^{(a)}_z(x,y,t) = \La_x P^{(a)}_z(x,y,t).
\]
\end{proposition}

We are ready to prove one of the central results in this note.

\begin{proof}[Proof of Theorem \ref{wk2}]  
Differentiating under the integral sign and using Proposition \ref{G-prop 9.6}  it is easily seen that $V$ satisfies the partial differential equation in \eqref{ep}. We thus turn to proving \eqref{tr2}. Taking the Fourier transform of $V$ with respect to the variable $t$, and using \eqref{hs}, \eqref{ptauH} and the spectral decomposition for $-\La$, we have the following alternative  representation for $V$, 
\begin{eqnarray}\label{ftv}
\widehat{V}(x,z,\sigma)& = &\frac{z^{1-a}}{2^{1-a}\Gamma\left(\frac{1-a}{2}\right)}\int^{\infty}_0\frac{e^{-\frac{z^2}{4\tau}}}{\tau^{\frac{3-a}{2}}}e^{-2\pi i\tau \sigma}P_{\tau}(\hat{u}(\cdot,\sigma))(x) d\tau\\
& = &\frac{z^{1-a}}{2^{1-a}\Gamma\left(\frac{1-a}{2}\right)}\int^{\infty}_0\big(\int^{\infty}_0\frac{e^{-\frac{z^2}{4\tau}}}{\tau^{\frac{3-a}{2}}}e^{-\big(\lambda+2\pi i\sigma\big)\tau} d\tau\big) dE(\lambda)(\hat{u}(\cdot,\sigma))(x).
\notag
\end{eqnarray}
Using Plancherel theorem in the $t$ variable we note that, 
\[
\|V(\cdot,z,\cdot)-u\|_{L^{2}\left(\R^{n+1}\right)} = \|\hat{V}(\cdot,z,\cdot)-\hat{u}\|_{L^{2}\left(\R^{n+1}\right)}.
\]
Since  for every $z>0$,
\begin{equation}\label{ident}
\frac{1}{2^{1-a}\Gamma\left(\frac{1-a}{2}\right)}z^{1-a}\int^{\infty}_0\frac{1}{t^{\frac{3-a}{2}}}e^{-\frac{z^2}{4t}} dt = 1,
\end{equation}
from \eqref{ftv} we infer that the following important identity 
\begin{align}\label{difvu}
& \hat{V}(x,z,\sigma)-\hat{u}(x,\sigma)
\\
& = \frac{z^{1-a}}{2^{1-a}\Gamma\big(\frac{1-a}{2}\big)}\int^{\infty}_0\big(\int^{\infty}_0\frac{e^{-\frac{z^2}{4\tau}}}{\tau^{\frac{3-a}{2}}}\big[e^{-\big(\lambda+2\pi i\sigma\big)\tau}-1\big] d\tau\big) dE(\lambda)(\hat{u}(\cdot,\sigma))(x).
\notag
\end{align}
Then,
\begin{align}\label{lh1}
& \|\hat{V}(\cdot,z,\cdot)-\hat{u}\|^2_{L^2(\R^{n+1})} 
\\
& = C(a)z^{2(1-a)}\int_{\R}\int_{0}^{\infty}\Big|\int^{\infty}_0\frac{e^{-\frac{z^2}{4\tau}}}{\tau^{\frac{3-a}{2}}}\Big[e^{-\left(\lambda+2\pi i\sigma\right)\tau}-1\Big]d\tau\Big|^2 d\|E(\lambda)\hat{u}(\cdot,\sigma)\|^2 d\sigma
\notag\\
& = C(a)z^{2(1-a)}\int_\R\int_{0}^{\infty}\Big|\int^{\infty}_0\frac{e^{-(\lambda+2\pi i\sigma)\frac{z^2}{4\theta}}}{\theta^{\frac{3-a}{2}}}\Big[e^{-\theta}-1\Big]d\theta\Big|^2|\lambda+2\pi i\sigma|^{1-a} d\|E(\lambda)\hat{u}(\cdot,\sigma)\|^2 d\sigma\notag\\
& \leq  C(a)z^{2(1-a)}\int_\R \int_{0}^{\infty}\left(\int^{\infty}_0\frac{1}{\theta^{\frac{3-a}{2}}}\Big[1-e^{-\theta}\Big]d\theta\right)^2|\lambda+2\pi i\sigma|^{1-a}d\|E(\lambda)\hat{u}(\cdot,\sigma)\|^2d\sigma\notag\\
& = C(a)z^{2(1-a)}\int_\R\int_{0}^{\infty}|\lambda+2\pi i\sigma|^{1-a} d\|E(\lambda)\hat{u}(\cdot,\sigma)\|^2d\sigma\notag.
\end{align}
We note that the second equality in \eqref{lh1} can be justified
by applying, for every $z>0$, $\la>0$ and $\sigma\in \R$ fixed, Cauchy's integral formula to the holomorphic function in the half-plane $\Re w>0$, 
 \[
f(w) =  \frac{e^{-\frac{z^2(\lambda + 2\pi i \sigma)}{4w}}(e^{-w}-1)}{w^{\frac{3-a}{2}}},
\]
and to the counterclockwise oriented contour $\Gamma_{\ve,R}$ formed by the two rays $\{(\lambda + 2\pi i \sigma )\tau \mid \tau>0\}$,  and $\{\Im w=0, \Re w>0\}$, and the portion of the circles $\{|w|^2 =R^2\}$ and $\{|w|^2 =\ve^2\}$ connecting them. Then, we let $\ve\to 0^+$ and $R \to \infty$. Noting that the integrals of $f(w)$ along the circular arcs of $\Gamma_{\ve,R}$ tend to zero as $\ve\to 0^+$ and $R\to \infty$, we conclude that 
\[
\int^{\infty}_0\frac{e^{-\frac{z^2}{4\tau}}}{\tau^{\frac{3-a}{2}}}\Big[e^{-\left(\lambda+2\pi i\sigma\right)\tau}-1\Big]d\tau = (\lambda+2\pi i\sigma)^{\frac{1-a}2} \int^{\infty}_0\frac{e^{-(\lambda+2\pi i\sigma)\frac{z^2}{4\theta}}}{\theta^{\frac{3-a}{2}}}\Big[e^{-\theta}-1\Big]d\theta, 
\]
which justifies the second equality in \eqref{lh1}.
 Consequently, keeping in mind that $1-a=2s$, from \eqref{lh1} we obtain
\begin{align*}
\|\hat{V}(\cdot,z,\cdot)-\hat{u}\|^2_{L^2(\R^{n+1})} 
& \leq  C(a)z^{2(1-a)}\int_\R \int^{\infty}_0|\lambda+2\pi i\sigma|^{2s} d\|E(\lambda)\hat{u}(\cdot,\sigma)\|^2 d\sigma 
\\
& =  C(a)z^{2(1-a)} \|\widehat{\cH^{s}u}(\cdot,\sigma)\|^2_{L^2(\R^n)},
\end{align*}
where in the last equality we have used \eqref{hes}. We infer that $V$ converges to $u$ in $L^2$ as $z\goto 0^+$.

We next prove \eqref{dn2}. Differentiating  \eqref{difvu} with respect to $z$, we find
\begin{align}\label{gud}
& \frac{\p \hat{V}}{\p z}(x,z,\sigma)
\\
& = \frac{z^{1-a}}{2^{1-a}\Gamma\big(\frac{1-a}{2}\big)}\int^{\infty}_0\big(\int^{\infty}_0\big(\frac{1-a}z - \frac{z}{2\tau}\big) \frac{e^{-\frac{z^2}{4\tau}}}{\tau^{\frac{3-a}{2}}}\big[e^{-\big(\lambda+2\pi i\sigma\big)\tau}-1\big] d\tau\big) dE(\lambda)(\hat{u}(\cdot,\sigma))(x).
\notag
\end{align}
Now we make the crucial observation that 
\begin{equation}\label{observe}
- \frac{d}{d\tau}\left(\frac{e^{-\frac{z^2}{4\tau}}}{\tau^{\frac{1-a}{2}}} \right)=  \frac z2 \frac{e^{-\frac{z^2}{4\tau}}}{\tau^{\frac{3-a}{2}}}\left(\frac{1-a}{z}-\frac{z}{2\tau}\right).
\end{equation}
Using \eqref{observe} in \eqref{gud}, we obtain the second basic identity
\begin{align}\label{verygud}
& \frac{\p \hat{V}}{\p z}(x,z,\sigma)
\\
& = - \frac{z^{-a}}{2^{-a}\Gamma\big(\frac{1-a}{2}\big)}\int^{\infty}_0\big(\int^{\infty}_0 \frac{d}{d\tau}\big(\frac{e^{-\frac{z^2}{4\tau}}}{\tau^{\frac{1-a}{2}}} \big) \big[e^{-\big(\lambda+2\pi i\sigma\big)\tau}-1\big] d\tau\big) dE(\lambda)(\hat{u}(\cdot,\sigma))(x)
\notag\\
& = \frac{z^{-a}}{2^{-a}\Gamma(\frac{1-a}{2})}\int^{\infty}_0\frac{e^{-\frac{z^2}{4\tau}}}{\tau^{\frac{1-a}{2}}}\,\frac{d}{d\tau}\left(\int^{\infty}_0[e^{-(\lambda+2\pi i \sigma)\tau}-1]\,dE(\lambda)(\hat{u}(\cdot,\sigma))(x)\right) d\tau\nonumber\\
& =-\frac{z^{-a}}{2^{-a}\Gamma(\frac{1-a}{2})}\int^{\infty}_0\frac{e^{-\frac{z^2}{4\tau}}}{\tau^{\frac{1-a}{2}}}\left(\int^{\infty}_0e^{-(\lambda+2\pi i \sigma)\tau}(\lambda+2\pi i \sigma) dE(\lambda)(\hat{u}(\cdot,\sigma))(x)\right) d\tau\nonumber\\
&  =-\frac{z^{-a}}{2^{-a}\Gamma(\frac{1-a}{2})}\int^{\infty}_0 (\lambda+2\pi i \sigma) \bigg(\int^{\infty}_0 \frac{e^{-\frac{z^2}{4\tau}}}{\tau^{\frac{1-a}{2}}}e^{-(\lambda+2\pi i \sigma)\tau} d\tau\bigg)  dE(\lambda)(\hat{u}(\cdot,\sigma))(x) \nonumber\\
& = -\frac{z^{-a}}{2^{-a}\Gamma(\frac{1-a}{2})}\int^{\infty}_0\left(\lambda+2\pi i \sigma\right)^{\frac{1-a}{2}}\left(\int^{\infty}_0\frac{e^{-\theta}e^{-(\lambda+2\pi i \sigma)\frac{z^2}{4\theta}}}{\theta^{\frac{1-a}{2}}} d\theta\right) dE(\lambda)(\hat{u}(\cdot,\sigma))(x).
\notag
\end{align}
Note that the last equality can be justified by applying Cauchy's integral theorem to the function 
\[
g(w) = \frac{e^{-\frac{z^2}{4w}(\la + 2\pi i \sigma)} e^{-w}}{w^{\frac{1-a}2}},
\]
and to the same path $\Gamma_{\ve,R}$ as in the proof of \eqref{lh1}. From \eqref{he} and \eqref{verygud} we thus find
\begin{eqnarray*}
  &&\left\|\frac{2^{-a}\Gamma\left(\frac{1-a}{2}\right)}{\Gamma\left(\frac{1+a}{2}\right)}z^a\frac{\pa \hat{V}}{\pa z}(\cdot,z,\cdot) + \widehat{\cH^{s}u}\right\|^2_{L^{2}\left(\R^{n+1}\right)}\\
  &=& \int_{\R}\int^{\infty}_0\left(\lambda+2\pi i\sigma\right)^{2s}\left(\frac{1}{\Gamma\left(\frac{1+a}{2}\right)}\int^{\infty}_0\frac{e^{-\theta}e^{-(\lambda+2\pi i \sigma)\frac{z^2}{4\theta}}}{\theta^{\frac{1-a}{2}}} d\theta-1\right)^2d\|E(\lambda)\hat{u}(\cdot,\sigma)\|^2d\sigma,
\end{eqnarray*}
and the right-hand side of the latter equality tends to zero as $z\to 0^+$ in view of Lebesgue dominated convergence and of the equation
$\frac{1}{\Gamma\left(\frac{1+a}{2}\right)}\int^{\infty}_0\frac{e^{-\theta}}{\theta^{\frac{1-a}{2}}} d\theta=1$.
By Plancherel theorem in $t$ variable we thus finally obtain 
\begin{align*}
& \lim_{z\rightarrow 0^+} \left\|\frac{2^{-a}\Gamma\left(\frac{1-a}{2}\right)}{\Gamma\left(\frac{1+a}{2}\right)}z^a\frac{\pa V}{\pa z}(\cdot,z,\cdot) + \cH^{s}u\right\|_{L^{2}\left(\R^{n+1}\right)} 
\\
& = \lim_{z\rightarrow 0^+}\left\|\frac{2^{-a}\Gamma\left(\frac{1-a}{2}\right)}{\Gamma\left(\frac{1+a}{2}\right)}z^a\frac{\pa \hat{V}}{\pa z}(\cdot,z,\cdot) + \widehat{\cH^{s}u}\right\|^2_{L^{2}\left(\R^{n+1}\right)} = 0,
\end{align*}
which gives the desired conclusion \eqref{dn2}.

We now prove our final claim \eqref{en2}.  First, we note from \eqref{extKernel} and \eqref{v} that
\begin{align}\label{to*}
&\bigg( \int_{\R} \int_{\R^n} V(x, z, t)^2 dx dt\bigg)^{1/2}
\\
& =  C(a) z^{2s} \bigg(\int_{\R} \int_{\R^n}  \left( \int_{0}^{\infty} \int_{\R^n}\frac{e^{-\frac{z^2}{4\tau}}}{\tau^{1+s}} p(x, y, \tau)u(y, t-\tau) dy d\tau \right)^2 dx dt\bigg)^{1/2}
\notag
\\
&  \leq  C(a) z^{2s}  \int_{0}^{\infty} \bigg(\int_{\R^n} \int_{\R} \int_{\R^n} p(x,y,\tau) u(y, t-\tau)^2  dxdt dy \bigg)^{1/2} \frac{e^{-\frac{z^2}{4\tau}}}{\tau^{1+s}} d\tau
\notag
\\
&  \leq   C(s) ||u||_{L^{2}(\R^{n+1})}.
\notag
\end{align}
Note that  the first inequality in \eqref{to*} follows by an application of  Minkowski and Jensen inequalities,  and the second one  follows from the fact that 
\[
\int_{\R^n} p(x, y,\tau) dx = 1,\quad \text{and}\  \int_{0}^{\infty}  \frac{z^{2s}}{\tau^{1+s}}  e^{-\frac{z^2}{4\tau}} d\tau =  4^s \Gamma(s).
\]
For $M>0$ let $\Sigma= \R^n \times (0, M) \times \R$. Since $a \in (-1, 1)$, by Fubini's theorem  we  easily obtain from \eqref{to*}
\begin{equation}\label{to1}
\|V\|_{L^2(\Sigma, z^a dx dz dt)} \leq  C(a, M) \|u\|_{L^2(\R^{n+1})}.
\end{equation}
Keeping \eqref{verygud} in mind, we find next
\begin{eqnarray}\label{ndvdz1}
&&\Big\|\frac{\pa \hat{V}}{\pa z}(\cdot,z,\cdot)\Big\|^2_{L^2(\R^n)}\\
& = &C(a)z^{-2a}\int^{\infty}_0|\lambda+2\pi i \sigma|^{1-a}\big|\int^{\infty}_0\frac{e^{-\theta}e^{-(\lambda+2\pi i \sigma)\frac{z^2}{4\theta}}}{\theta^{\frac{1-a}{2}}} d\theta\big|^2 d\|E(\lambda)\hat{u}(\cdot,\sigma)\|^2\nonumber.
\end{eqnarray}
From \eqref{ndvdz1} we thus obtain,
\begin{align}\label{tuo}
&\Big\|\frac{\pa \hat{V}}{\pa z}(\cdot,z,\cdot)\Big\|^2_{L^2(\Sigma, z^a \,dz\, dx\, dt)}\\
&\leq\  C(a, M)\int^{\infty}_{-\infty}\int^{\infty}_0|\lambda+2\pi i \sigma|^{2s}\,d\|E(\lambda)\hat{u}(\cdot,\sigma)\|^2\,d\sigma \leq C(a, M) ||u||_{\mathcal{H}^{2s}},\notag
\end{align} 
where in the last inequality we have used \eqref{H2s}. 
We now note that  for an arbitrary function $f$, appropriately decaying at infinity in $\R^{n}$, an integration by parts gives  
\[
\|Xf\|_{L^2(\R^n)} = \|\Ls^{\frac{1}{2}}f\|_{L^2(\R^n)}.
\]
Combining this observation with Plancherel theorem in the $t$ variable, we find  
\begin{equation}\label{equi1}
\|XV(\cdot,z,\cdot)\|_{L^2(\R^n \times \R)} =  \|\Ls^{\frac{1}{2}}\widehat{V}(\cdot,z,\cdot)\|_{L^2(\R^n \times \R)}.
\end{equation}
Using \eqref{fracsl} with $s = 1/2$, and the representation \eqref{ftv}, we obtain
\begin{align*}
&\|\Ls^{\frac{1}{2}}\widehat{V}(\cdot,z,\sigma)\|^2_{L^2(\R^n)}
\leq C(a)z^{2(1-a)}\int^{1}_0\lambda\bigg|\int^{\infty}_0\frac{e^{-\frac{z^2}{4\tau}}}{\tau^{\frac{3-a}{2}}}e^{-(\lambda+2\pi i \sigma)\tau} d\tau\bigg|^2 d\|E(\lambda)\hat{u}(\cdot,\sigma)\|^2
\\
&+C(a)z^{2(1-a)}\int^{\infty}_1\lambda\bigg|\int^{\infty}_0\frac{e^{-\frac{z^2}{4\tau}}}{\tau^{\frac{3-a}{2}}}e^{-(\lambda+2\pi i \sigma)\tau} d\tau \bigg|^2 d\|E(\lambda)\hat{u}(\cdot,\sigma)\|^2 = I(z,\sigma) + II(z,\sigma).
\end{align*}
We now estimate
\begin{align*}
I(z,\sigma) & \le C(a)z^{2(1-a)}\int^{1}_0\lambda\bigg(\int^{\infty}_0\frac{e^{-\frac{z^2}{4\tau}}}{\tau^{\frac{1-a}{2}}}e^{-\lambda\tau} \frac{d\tau}\tau \bigg)^2 d\|E(\lambda)\hat{u}(\cdot,\sigma)\|^2
\\
& \le C(a)z^{2(1-a)}\int^{1}_0 \bigg(\int^{\infty}_0\frac{e^{-\frac{z^2}{4\tau}}}{\tau^{\frac{1-a}{2}}} \frac{d\tau}\tau \bigg)^2 d\|E(\lambda)\hat{u}(\cdot,\sigma)\|^2
\\
& \le C(a)z^{2(1-a)} \bigg(\int^{\infty}_0\frac{e^{-\frac{z^2}{4\tau}}}{\tau^{\frac{1-a}{2}}} \frac{d\tau}\tau \bigg)^2 \int^{\infty}_0  d\|E(\lambda)\hat{u}(\cdot,\sigma)\|^2
\\
& = C'(a)\ \|\hat{u}(\cdot,\sigma)\|^2_{L^2(\Rn)},
\end{align*}
where the last equality is justified by the change of variable $\rho = \frac{z^2}{4\tau}$, and by the fact that $\Gamma(\frac{1-a}2) = \int_0^\infty \rho^{\frac{1-a}2} \frac{d\rho}\rho$. Next, an application of Cauchy's integral formula as in \eqref{lh1}, \eqref{verygud}, gives
\begin{align*}
II(z,\sigma) & = C(a)z^{2(1-a)}\int^{\infty}_1\lambda|\lambda+2\pi i \sigma|^{1-a}\bigg|\int^{\infty}_0\frac{e^{-(\lambda+2\pi i \sigma)\frac{z^2}{4\theta}}}{\theta^{\frac{3-a}{2}}}e^{-\theta} d\theta\bigg|^2 d\|E(\lambda)\hat{u}(\cdot,\sigma)\|^2
\\
& \le C(a)z^{2(1-a)}\int^{\infty}_1\lambda|\lambda+2\pi i \sigma|^{1-a}\left(\int^{\infty}_0\frac{e^{-\frac{\lambda z^2}{4\theta}}}{\theta^{\frac{1-a}{2}}}e^{-\theta} \frac{d\theta}\theta\right)^2 d\|E(\lambda)\hat{u}(\cdot,\sigma)\|^2.
\end{align*}
We now use the following formula that can be found in 9.   on p. 340 of \cite{GR}
\[
\int_0^\infty t^{\nu-1} e^{-(\frac{\beta}{t} + \gamma t)} dt = 2 \left(\frac \beta{\gamma}\right)^{\frac \nu{2}} K_\nu(2 \sqrt{\beta \gamma}),
\]
provided $\Re \beta, \Re \gamma >0$. Applying it with
$\nu = - \frac{1-a}2$, $\beta = \frac{\la z^2}4$ and $\gamma = 1$,
and keeping in mind that $K_\nu = K_{-\nu}$ (see 5.7.10 in \cite{Le}), we find
\[
\bigg(\int^{\infty}_0\frac{e^{-\frac{\lambda z^2}{4\theta}}}{\theta^{\frac{1-a}{2}}}e^{-\theta} \frac{d\theta}\theta\bigg)^2 = 4^{2-a} z^{-2(1-a)}\la^{-1+a} K_{\frac{1-a}2}(z\sqrt \la)^2.
\]
This gives
\begin{align*}
II(z,\sigma) \le C'(a) \int^{\infty}_1\lambda^a |\lambda+2\pi i \sigma|^{1-a} K_{\frac{1-a}2}(z\sqrt \la)^2 d\|E(\lambda)\hat{u}(\cdot,\sigma)\|^2.
\end{align*}
Combining estimates, we thus obtain 
\begin{align*}
& \|\Ls^{\frac{1}{2}}\widehat{V}(\cdot,z,\cdot)\|^2_{L^2(\Sigma, z^a dz dx dt)} \le C(a,M)\ \|\hat{u}\|^2_{L^2(\R^{n+1})}
\\
& + C'(a) \int^{\infty}_{-\infty}\int^{\infty}_1\la^a |\lambda+2\pi i \sigma|^{1-a}\int^{M}_0  K_{\frac{1-a}2}(z\sqrt \la)^2 z^a dz d\|E(\lambda)\hat{u}(\cdot,\sigma)\|^2 d\sigma.
\end{align*}
Observe now that, in view of the asymptotic behaviour of  $K_{\frac{1-a}2}(\rho)^2 $, see e.g. (5.11.9) in \cite{Le}, we have
\[
\int^{M}_0  K_{\frac{1-a}2}(z\sqrt \la)^2 z^a dz = \la^{-\frac{a+1}2} \int_0^{M \sqrt \la} K_{\frac{1-a}2}(\rho)^2 \rho^a d\rho \le C''(a) \la^{-\frac{a+1}2}.
\]
Substituting this information in the above inequality, we thus find
\begin{align*}
& \|\Ls^{\frac{1}{2}}\widehat{V}(\cdot,z,\cdot)\|^2_{L^2(\Sigma, z^a dz dx dt)} \le C(a,M)\ \|\hat{u}\|^2_{L^2(\R^{n+1})}
\\
& + C'''(a) \int^{\infty}_{-\infty}\int^{\infty}_1\la^{\frac{a-1}2} |\lambda+2\pi i \sigma|^{1-a} d\|E(\lambda)\hat{u}(\cdot,\sigma)\|^2 d\sigma
\\
& \le C'''(a) \int^{\infty}_{-\infty}\int^{\infty}_1|\lambda+2\pi i \sigma|^{1-a} d\|E(\lambda)\hat{u}(\cdot,\sigma)\|^2 d\sigma \le \tilde C(a) ||u||_{\mathcal H^{2s}},
\end{align*}
where in the last inequality we have used \eqref{H2s}
 and the fact that $1-a = 2s$.
Finally,  by taking \eqref{equi1} into account, we  obtain 
\begin{equation}\label{to2}
\| XV(\cdot,z,\cdot)\|^2_{L^2(\Sigma)} \leq C(a, M)\|u\|^2_{\mathcal{H}^{2s}}.
\end{equation}
The desired conclusion \eqref{en2} now follows from the estimates \eqref{to1}, \eqref{tuo} and \eqref{to2}. 

\end{proof}

We close this section by briefly discussing the extension problem for $(-\La)^s$. The relevant result is Theorem \ref{wk1} below, whose proof we do not present since it follows from that of Theorem \ref{wk2} and from Bochner's subordination.
For a given $s \in (0,1)$, let $a= 1-2s$. Now given $u\in W^{2s}$, let
\begin{eqnarray}
\label{U}
U(x,z)&\ =\ &
\int_{\R^n}\left(\int_0^\infty P^{(a)}_z(x,y,t)\,dt\right)\, u(y)\,dy \\
\nonumber
&\ =\ &
\frac{z^{1-a}}{2^{1-a}\Gamma\left(\frac{1-a}{2}\right)}\,\int_{\R^n}\left(\int^{\infty}_0\frac{e^{-\frac{z^2}{4t}}}{t^{\frac{1-a}{2}}}\,p(x,y,t)\,\frac{dt}{t}\right)\,u(y)\,dy\\
\nonumber
&\ =\ &\frac{z^{1-a}}{2^{1-a}\Gamma\left(\frac{1-a}{2}\right)}\,\int^{\infty}_0\frac{e^{-\frac{z^2}{4t}}}{t^{\frac{1-a}{2}}}\,P_tu(x)\,\frac{dt}{t}\\
\nonumber
&\ =\ &\frac{z^{1-a}}{2^{1-a}\Gamma\left(\frac{1-a}{2}\right)}\int^{\infty}_0\frac{e^{-\frac{z^2}{4t}}}{t^{\frac{1-a}{2}}}\left(\int^{\infty}_0e^{\lambda t}\,dE(\lambda)u\right)\,\frac{dt}{t}.
\end{eqnarray}  
Differentiating under the integral sign and  by  arguing as in the proof of  Proposition   10.4 in \cite{G} we have that 
\begin{equation}\label{strong}
\La_a U \overset{def}{=} z^{a}\left(\La U+\B_aU\right) = 0,\ \ \ z >0.
\end{equation}
The operator $\La_a$ can be written as
\begin{equation}
\label{extdivU}
\La_a  = -\sum^{m+1}_{i=1}\tilde{X}^{\star}_i\left(\omega(x,z)\tilde{X}_i\right),
\end{equation}
where $\tilde{X}_i$, $i=1,...,m+1$ and $\omega$ are as in \eqref{tildes}. 
The fact that $U$ defined above is a  weak solution to  the following boundary  problem 
\begin{equation}\label{wk}
\begin{cases}
\La_a U=0\ \text{in}\ \{z>0\},
\\
U(x, z)=u(x)\ \text{in}\ L^{2}(\R^n),
\end{cases}
\end{equation}
and that $(-\La)^s$ can be realized as a weighted Dirichlet-to-Neumann map corresponding to the operator  $\La_a$, is a consequence of the following theorem. 

\begin{theorem}\label{wk1}
For a given $u \in W^{2s}$, let $U$ be as in \eqref{U}. Then, $U$ solves the partial differential equation in \eqref{wk}. Furthermore, we have 
\begin{equation}\label{tr1}
\lim_{z\rightarrow 0^+}\|U(\cdot,z)-u\|_{L^{2}\left(\R^{n}\right)}=0,
\end{equation}
and
\begin{equation}\label{dn1}
-\frac{2^{-a}\Gamma\left(\frac{1-a}{2}\right)}{\Gamma\left(\frac{1+a}{2}\right)}\lim_{z\rightarrow 0^+}z^a\frac{\pa U}{\pa z}(x,z) = \Ls^{s}u(x)\ \text{in $L^{2}(\R^n)$}.
\end{equation}
Finally, for every $M>0$ and $\Sigma= \R^n \times (0, M)$, one has
\begin{equation}\label{en1}
\|U\|_{\cL^{1,2}(\Sigma, z^a dx dz)} \leq  C\|u\|_{2s},
\end{equation}
for some $C=C(a,M)>0$. 
\end{theorem}


\section{Proof of the Harnack inequality for $\cH^s$}\label{final}

As it has been mentioned in the introduction, the Harnack inequality for  $\cH^s$ and $\Ls^s$ is derived from that of the corresponding degenerate local operators $\cH_a$ and $\La_a$ in \eqref{extdiv} and \eqref{extdivU} respectively. Since the Harnack inequality for the time-independent operator $\La_a$ descends easily from that of the time-dependent operator $\cH_a$, we focus on this one. From the works \cite{Gry} and \cite{Sa} it is known that the Harnack inequality for parabolic operators generated by smooth Dirichlet forms is equivalent to the doubling condition and the Poincar\'e inequality for the corresponding metric balls  associated with the \emph{carr\'e du champ}. Now, because of the presence of the non-smooth weight $|z|^a$ the extended operator in \eqref{extdiv} does not satisfy the structural assumptions in \cite{Gry}, \cite{Sa}. Nevertheless, the ideas in those papers, or more precisely the arguments in \cite{CSe}, do apply once we have the appropriate doubling condition and the Poincar\'e inequality in the extended space.
We thus turn to verify the validity of these properties for the metric balls corresponding to the extended system of vector fields \eqref{tildes}.

For a given $(x,z) \in \Rn \times \R$ and for $r>0$ we consider the cylinders
\[
C_r(x,z)= B(x,r) \times (z-r, z+r),
\]
where $B(x,r)$ is the ball in the control distance \eqref{d} corresponding to $\La$. Then, the doubling property
\begin{equation}\label{doubling1}
|C_{2r}(x,z)| \leq C\ |C_r(x,z)|
\end{equation}
follows immediately from \eqref{doubling}.  We now show that the doubling property for the metric balls follows from \eqref{doubling1}. 
Observe that  the  \emph{carr\'e du champ} relative to  the extended system $\tilde X$ is
 \[
|\tilde X u|^2 = |Xu|^2 + (\partial_z u)^2.
\]
Since we have assumed that $d$ is in fact a distance, and $\partial_z$ is independent of $\La$, it easily follows that 
 \begin{equation}\label{td}
\tilde d((x,z),(x',z')) = \sup\{|\vf(x,z) - \vf(x',z')|\mid \vf\in C^\infty(\R^{n+1}),\ |\tilde X\vf|\le 1\}
\end{equation}
is also a distance. 
We also infer from  \eqref{td} that there exist constants $0<\sigma_1 < \sigma_2$ such that 
\begin{equation}\label{comp1}
\tilde B((x,z),\sigma_1 r) \subset C_r(x,z) \subset \tilde B((x,z), \sigma_2 r),
\end{equation}
where $\tilde B((x,z),r)$ denotes the  ball relative to $\tilde d$ centred at $(x,z)$ and of radius $r$. From \eqref{doubling1} and \eqref{comp1} we easily conclude the validity of  
\begin{equation}\label{db1}
|\tilde B((x,z), 2r)| \leq C |\tilde B((x,z), r)|,
\end{equation}
for a new constant $C$ depending also on $\sigma_1$ and $\sigma_2$.

We next show  that the following Poincar\'e inequality holds,
\begin{equation}\label{poi1}
\int_{\tilde B((x,z), r)}  (f - f_{\tilde B((x,z), r)})^2  |z|^a dxdz \leq  C_1 r^2  \int_{\tilde B((x,z),  2r)}  |\tilde Xf|^2 |z|^a dxdz,
\end{equation}
where
\[
f_{\tilde B((x,z), r)}= \frac{\int_{\tilde B((x,z), r)}  f\  |z|^a dxdz}{\int_{\tilde B((x,z), r)} |z|^a dxdz}.
\]
To prove \eqref{poi1} we observe that, on one hand, \eqref{poi} is valid by hypothesis. On the other hand, since $\omega(z) = |z|^a$ is an $A_2$-weight on the real line, the following one-dimensional Poincar\'e inequality holds (see for instance \cite{FKS}),
\begin{equation}\label{1d}
\int_{(z-r, z+r)}  (f- f_{z,r})^2 |z|^a dz \leq C \int_{(z-r, z+r)}  f'(z)^2 |z|^a dz,
\end{equation}
where
\[
f_{z,r}= \frac{\int_{(z-r, z+r)}  f |z|^a dz}{\int_{(z-r, z+r)} |z|^a dz}.
\]
We can thus apply the real analysis argument in the proof of \cite[Lemma 2]{LW}  to show that the Poincar\'e inequality holds with respect to the  cylinders $\{C_r(x,z)\}$, i.e.,
\begin{equation}\label{poi2}
\int_{ C_r(x,z)} (f - f_{ C_r(x,z)})^2  |z|^a  dxdz \leq  C_1 r^2  \int_{ C_{2r}(x,z)} |\tilde Xf|^2 |z|^a  dxdz,
\end{equation}
where
\[
f_{ C_r(x,z)}= \frac{\int_{C_r(x,z)} f  |z|^a  dxdz}{\int_{C_r(x,z)} |z|^a dxdz}.
\]
With $\omega(\tilde B((x,z), r)) = \int_{\tilde B((x,z), r)} |z|^a dz dx$, it also follows from  \eqref{comp1} that the following weighted doubling condition holds,
\begin{equation}\label{db2}
\omega(\tilde B((x,z), 2r)) \leq C_1 \omega(\tilde B((x,z), r)).
\end{equation}
Thus, using \eqref{comp1}, \eqref{poi2}
 and \eqref{db2}, by a standard argument (see the proof of  \cite[Theor. 1]{HK}, and also \cite{GN}, \cite{J}),  we conclude that \eqref{poi1} holds.  Moreover, arguing as in the proof of \cite[Theor. 2.1]{GN}, or as in \cite[Theor. 1]{HK}, from \eqref{poi1} and \eqref{db2} we infer that the following Sobolev type inequality holds for all $0<r \leq 1$, and every $\phi \in C_0^{0,1}(\tilde B((x,z),r))$, 
\[
\bigg(\int_{\tilde B((x,z), r)} |\phi|^{2\kappa} |z|^a dx dz \bigg)^{\frac{1}{2\kappa}}  \leq C \bigg(\int_{\tilde B((x,z), r)} |X \phi|^2 |z|^a  dx dz \bigg)^{1/2},
\]
for some $\kappa >1$ depending  on the  doubling constant  in \eqref{db2}.
At this point, we can apply the circle of ideas in \cite{Gry} and  \cite{Sa}  to deduce that the following Harnack inequality holds for the degenerate parabolic operator $\cH_a$ in \eqref{extdiv}.

\begin{theorem}\label{harext}
Let $U$ be a solution to $\cH_a U=0$  in $C_r(x,z) \times (-r^2, 0]$ with $\cH_a$ as in \eqref{extdiv}. Then, the following scale invariant Harnack inequality holds,
\begin{equation}\label{sc1}
\sup_{C_{r/2} (x,z) \times \{t= -\frac{r^2}{2}\} } U \leq C_0 \inf_{C_{r/2} (x,z) \times \{t= 0\} } U,
\end{equation}
for some universal constant $C_0>0$ depending on the constants in \eqref{poi1} and \eqref{db2}. 
\end{theorem}

\begin{remark}
We mention that, when the system $\tilde X$ is of H\"ormander type, the Poincar\'e inequality \eqref{poi1} follows from Lu's work \cite{Lu}.
\end{remark}

\begin{remark}
We also note that given the validity of  \eqref{db1} and \eqref{poi1}, it follows from Theorem 4.4 in \cite{Sa} that  the operator $\cH_a$ is hypoelliptic  for $z>0$. 

\end{remark}

With Theorem \ref{harext} in hands, we are finally able to provide the   

\begin{proof}[Proof of Theorem \ref{parharnack}]

It follows from Theorem \ref{wk2} that, given a function $u\in H^{2s}$, the function $V$ defined in \eqref{v} is a weak solution to 
\begin{equation}
\begin{cases}
\cH_a V = 0 \ \ \ \ \ \ \text{for}\ z>0,
\\
V(x,0, t) =  u(x,t),
\\
\underset{z \to 0^+}{\lim} z^a \partial_z V  =  0\ \text{in}\ B(x, 2r) \times (-2r^2, 0],
\end{cases}
\end{equation}
in the sense of Definition \ref{wk0} where $\cH_a$ is  the extension operator in \eqref{extdiv}. 
Now thanks  to the  vanishing Neumann condition, if $V$ is evenly reflected,  it follows by a standard  argument that  the extended  $V$  is  a weak solution to $\cH_a V=0$ in $C_{2r}((x, 0)) \times (-2r^2, 0]$.    At this point,   we can apply the scale invariant Harnack inequality \eqref{sc1} to $V$  and subsequently by  arguing as in the proof of Theorem 4.8 in \cite{FF} we  obtain the desired conclusion.

\end{proof}

\end{document}